\documentclass{article}
\usepackage[letterpaper,margin=1.0in]{geometry}
\usepackage{amsthm}

\usepackage{amsmath,amsfonts,bm,amssymb}

\def\eqref#1{equation~\ref{#1}}

\def\1{\bm{1}}

\def\eps{{\epsilon}}

\def\vzero{{\bm{0}}}

\def\vu{{\bm{u}}}
\def\vv{{\bm{v}}}

\def\vx{{\bm{x}}}
\def\vy{{\bm{y}}}

\def\mF{{\bm{F}}}
\def\mG{{\bm{G}}}

\def\mP{{\bm{P}}}

\DeclareMathAlphabet{\mathsfit}{\encodingdefault}{\sfdefault}{m}{sl}
\SetMathAlphabet{\mathsfit}{bold}{\encodingdefault}{\sfdefault}{bx}{n}

\def\gN{{\mathcal{N}}}
\def\gO{{\mathcal{O}}}

\def\gT{{\mathcal{T}}}
\def\gU{{\mathcal{U}}}
\def\gV{{\mathcal{V}}}

\def\gX{{\mathcal{X}}}

\def\sR{{\mathbb{R}}}

\newtheorem{thm}{Theorem}[section]
\newtheorem{dfn}{Definition}[section]

\newtheorem{lem}{Lemma}[section]
\newtheorem{asm}{Assumption}[section]

\newtheorem{prop}{Proposition}[section]

\usepackage{hyperref}
\usepackage{enumitem}
\usepackage[utf8]{inputenc} 
\usepackage[T1]{fontenc}    
\usepackage{hyperref}       
\usepackage{url}            
\usepackage{booktabs}       
\usepackage{amsfonts}       
\usepackage{nicefrac}       
\usepackage{microtype}      
\usepackage{graphicx}       
\usepackage[numbers,sort]{natbib}
\usepackage{multirow}
\usepackage{bbding}
\graphicspath{{media/}}     
\usepackage{color}
\usepackage{algpseudocode,algorithm}
\usepackage{nicefrac}
\usepackage{cancel}
\usepackage{graphicx} 
\usepackage{booktabs}
\usepackage{makecell}
\usepackage{multirow}
\usepackage{amssymb}

\usepackage{thmtools}
\usepackage{thm-restate}

\title{Halpern Iteration Accelerates Solving High-Order Monotone Variational Inequalities}

\author{Lesi Chen \textsuperscript{* 1} \qquad
Xinliang Zhang \textsuperscript{* 1} \qquad Hengyu Wang \textsuperscript{* 3 4} \\
Chengchang Liu \textsuperscript{5} \qquad
Yongchao Chen \textsuperscript{2 3}
\qquad Jingzhao Zhang \textsuperscript{1 2} \\
\vspace{-2mm} \\
\normalsize{\textsuperscript{1}IIIS, Tsinghua University\quad 
\textsuperscript{2} College of AI, Tsinghua University} \quad \textsuperscript{3} Apex Intelligence \\
\normalsize{
\textsuperscript{4} School of Mathematical Sciences, Tongji University} \\
\normalsize{
\textsuperscript{5} Department of Artificial Intelligence, Westlake University} \\
\vspace{-2mm} \\
\normalsize{ \texttt{ \{chenlc23, xinliang23@\}@mails.tsinghua.edu.cn, wanghengyu@apexin.ai}} \\
\normalsize{\texttt{
liuchengchang@westlake.edu.cn,
cyc@apexin.ai,
jingzhaoz@mail.tsinghua.edu.cn
}}}

\begin{document}
\maketitle
\begingroup
\begin{NoHyper}
\renewcommand\thefootnote{*}
\footnotetext{Equal contributions.}
\end{NoHyper}
\endgroup
\begin{abstract}
In this short note, we revisit the classical Halpern iteration in 1967 and show its new applications in designing second- and higher-order methods for solving smooth monotone variational inequalities (MVI). By using a large-step inexact Halpern iteration, we propose a novel $p$th-order method Halpern-ATM, that achieves a fast convergence rate of $\tilde \gO(T^{-p})$. This improves the best-known results for $p \ge 2$ since 2012.

\end{abstract}

\section{Introduction}

In this note, we study the monotone variational inequality (MVI) problem \citep{facchinei2003finite}, which targets at finding an solution $\vx^* \in \gX$ such that
\begin{align} \label{eq:prob-mvi}
    \langle F(\vx^*), \vx - \vx^* \rangle \ge0 \quad \forall \vx \in \gX,
\end{align}
where $\gX \subseteq\mathbb{R}^{d}$ is a nonempty compact convex set, and \(F:\gX \to \sR^d\) is a continuously differentiable monotone operator. The MVI problem covers the computation of many game-theoretical equilibria \citep{kinderlehrer2000introduction,giannessi1995variational,nocedal1999numerical,ba2025doubly,jordan2025adaptive}. An important special case of the MVI problem, the convex-concave minimax optimization problem $\min_{\vu \in \gU} \max_{\vv \in \gV} \phi(\vu, \vv)$ is equivalent to the MVI problem for the gradient operator $\mF(\vu,\vv) = [\nabla_\vu \phi(\vu,\vv), - \nabla_\vv \phi(\vu,\vv)]$, which underpins influential machine-learning applications such as adversarial training \citep{zhang2018mitigating} and distributionally robust optimization~\citep{carmon2022distributionally}. 

A classical approach to solving an MVI is the \textit{proximal point method} \citep{rockafellar1976monotone}, which iteratively computes proximal point $\vx_{t+1} = \mP_\eta(\vx_t)$, where $\mP_\eta = ({\rm Id} + \eta (\mF + \gN_\gX)))^{-1}$ is the resolvent operator and $\gN_\gX$ is the normal-cone operator. The state-of-art first and high-order methods can all be obtained from the PPM framework \citep{mokhtari2020unified,jiang2022generalized}, including Korpelevich's first-order extragradient (EG) method \citep{korpelevich1976extragradient},  Monteiro-Svaiter's Newton Proximal Extragradient (NPE) method \citep{monteiro2012iteration} and its high-order generalizations \citep{bullins2022higher,adil2022optimal,huang2022approximation,lin2023monotone,jiang2022generalized,lin2022perseus,nesterov2023high}.

The \textit{Halpern iteration} $\vx_{t+1} = \beta_t \vx_0 + (1- \beta_t) \mP_\eta(\vx_t)$ with $\beta_t = 1/(t+2)$ \citep{halpern1967fixed} adds a pull toward the initial point $\vx_0$ per iteration, which gives an alternative to achieving the same convergence rate as PPM for computing fixed points. This iteration has recently been studied for first-order methods, resulting in the anchored extragradient methods \citep{diakonikolas2020halpern,yoon2021accelerated,cai2022finite} that achieve the optimal $\mathcal{O}(T^{-1})$ convergence rate on the proximal residual ${\rm res}(\vx) = \| \vx - \mP_{\eta(\mF+\gN_\gX)}(\vx) \| / \eta$. This note discovers a new application of Halpern iteration in high-order methods. \textit{Our observation is that a smaller proximal residual also makes high-order proximal subproblems cheaper to solve, leading to new convergence rates for high-order methods.} Note that $R_t = \eta {\rm res}(\vx) = \| \vx - \mP_{\eta(\mF+\gN_\gX)}(\vx) \|$ is exactly the distance from the starting point to the solution of the $t$th subproblem. Although the complexity of first-order methods is independent of it, $R_t$ is in fact a key quantity in the complexity bounds of high-order methods \citep{monteiro2012iteration,huang2022approximation,jiang2022generalized}. For instance, for second-order methods $p=2$, the dependency is $R_t^{2/3}$. Therefore, a straightforward application of the Halpern iteration in the high-order setting leads to lower inner proximal solver costs, which finally improves the best-known complexity from $\mathcal{O}(T^{-(p+1)/2})$  \citep{monteiro2012iteration,bullins2022higher} to  $\tilde{\mathcal{O}}(T^{-p})$.

Within a month following our release, three independent works \citep{zhang2026optimal,zhang2026matching,jiang2026anchored} built upon our ideas to establish an optimal convergence rate of $\mathcal{O}(T^{-(3p+1)/2})$, matching the lower bound given by \citet{chen2026solving}. The core idea behind these developments is to incorporate an additional gradient step into the Halpern iteration to predict the optimal solution of each proximal subproblem, thereby further reducing the subproblem's computational complexity. While the derivations in those subsequent works are quite involved, this note provides a concise proof for improving the classical high-order algorithms using the classical Halpern iteration. For instance, in the case of $p=2$, our analysis largely reuses existing tools, making the entire proof of Theorem \ref{thm:main-p-2} only require just a few lines.

\paragraph{AI Usage.} 
This note interprets an initial AI response first discovered by Apex Intelligence through GPT-5.6 Sol (see Appendix \ref{sec:alg-AI}), while simplifying the resulting algorithm and providing a human-readable proof. We hope this note deepens the community's understanding of this fundamental problem.

\paragraph{Notations.}  We use $\Vert \, \cdot \, \Vert$ to denote the Euclidean norm for vectors and the spectral norm for matrices and tensors in a unified way.
We hide logarithmic factors in the notation $\tilde \gO(\,\cdot\,)$. Also, we use the notations $\gO_p(\,\cdot\,)$ and $\tilde \gO_p(\,\cdot\,)$ to hide the costants that depend on $p$. We denote $D^q \mF$ as the $q$th-order derivative of an operator $\mF: \gX \to \sR^d$. For a convex compact set ${\mathcal{X}}$, we let $\mathcal{N}_{\mathcal{X}}({\boldsymbol{x}})$ be its normal cone at point ${\boldsymbol{x}}$, equivalently, the subdifferential of the indicator of the set. We also let ${\rm Proj}_{\gX}$ be the projection operator of $\vx$ onto the set $\gX$.

\section{Preliminaries}

\subsection{Main Assumptions}

We study the MVI Problem (\ref{eq:prob-mvi}) under the following standard assumptions \citep{huang2022approximation,lin2022perseus,bullins2022higher,jiang2022generalized}.

\begin{asm} \label{asm:X}
$\gX \subseteq\mathbb{R}^{d}$ is a nonempty compact convex set.
\end{asm}

\begin{asm} \label{asm:x-star}
 We assume there exists $\vx^{\star},$ such that $ 0 \in (\mF + \gN_\gX)(\vx^{\star})$.
\end{asm}

\begin{asm} \label{asm:F}
    \(\mF:\gX \to\mathbb{R}^{d}\) is continuous and monotone:
\[
\langle \mF(\vx)-\mF(\vy),\vx-\vy\rangle\geq0,\quad \forall \vx,\vy\in \gX.
\]
\end{asm}
\begin{asm} \label{asm:pth-smooth}
Assume $F: \gX \to \sR^d$ is $p$th-order $L_p$ smooth:
\begin{equation} \label{eq:F-pth-smooth}
    \left\|D^{p-1}\mF(\vx)-D^{p-1}\mF(\vy)\right\|\leq L_{p}\left\|\vx-\vy\right\|, \quad \forall \vx,\vy \in \gX.
\end{equation}
\end{asm}


\subsection{High-Order Oracles and Tensor Steps}

We study algorithms that use the  $p$th-order oracles $(\mF(\vx), D \mF(\vx),\ldots, D^{p-1} \mF(\vx))$, and 
a basic operation to leverage this oracle is the following tensor step \citep{huang2022approximation,lin2022perseus,bullins2022higher,jiang2022generalized}, which solves the MVI problem/monotone inclusion induced by a local Taylor approximation.

\begin{dfn} \label{dfn:mvi-tensor-step}
Under Assumption \ref{asm:F} and \ref{asm:pth-smooth}, for an input point $\vx \in \gX$ the $p$th-order tensor step with regularization parameter $M\ge L_p$ outputs $\vy = \gT_\mF^p(\vx;M)$ such that
\begin{equation} \label{eq:first-order-tensor}
    \vzero \in \bar \mF_\vx^p(\vy)+\frac{M}{p!}\left\|\vy-\vx\right\|^{p-1}(\vy-\vx)+\gN_{\gX}(\vy),
\end{equation}
where \(\bar \mF_{\vx}^p(\vy):=\sum_{k=0}^{p-1}D^{k}\mF(\vx)[\vy-\vx]^{k} / k!\) is $(p-1)$th-order Taylor expansion for $\mF$ at the center point $\vx$.
\end{dfn}


\subsection{Resolvent Operators and Residuals}

To apply the Halpern iteration to MVI problems, we make use of the following 
resolvent/proximal operator \citep{rockafellar1976monotone,combettes2018monotone,facchinei2003finite}. 

\begin{dfn} \label{dfn:resolvent}
Under Assumption \ref{asm:F}, for $\eta>0$, we can define the unique-valued resolvent operator as
\[
\mP_\eta(\vx) = ({\rm Id} + \eta (\mF + \gN_\gX ))^{-1} (\vx).
\]
\end{dfn}
It is well-known that $\mP_\eta$ is non-expansive  if $\mF$ is monotone \citep{ryu2016primer,rockafellar1976monotone}.

\begin{prop}[{\citet[Section 6]{ryu2016primer}}]\label{prop:resolvent-geom}
Under Assumption \ref{asm:F}, the operator $\mP_\eta$ is non-expansive for any $\eta>0$, 
\[
\| \mP_{\eta}(\vx)-\mP_{\eta}(\vx') \| \le \| \vx - \vx' \|, \quad \forall \vx, \vx' \in \gX.
\]
\end{prop}

Moreover, it is easy to see that $\vx^*$ solved the MVI problem, \textit{i.e.}, $\vzero \in \mF(\vx^*) + \gN_\gX(\vx^*)$, if and only if $\vx^*$ is the fixed point of $\mP_\eta$, \textit{i.e.}, $\vx^* = \mP_\eta(\vx^*)$. It motives the definition of the following proximal residual \citep{diakonikolas2020halpern,cai2024variance,alacaoglu2024revisiting}.

\begin{dfn} \label{dfn:proximal-residual}
Under Assumption \ref{asm:F}, we define the proximal residual for a point $\vx \in \gX$ as
\[
{\rm res}(\vx) := \| \vx - \mP_\eta(\vx) \| / \eta.
\]
\end{dfn}
In addition, finding a point with a small proximal residual is sufficient for approximating the MVI problem, where the left-hand side below is the strong optimality criterion for MVIs \citep{cai2022finite,cai2023accelerated,cai2024accelerated}.

\begin{prop}[{\citet[Proposition 4.1]{cai2024variance}}]
Under Assumption \ref{asm:F},  then for $\vy = \mP_\eta(\vx)$, 
\begin{equation} \label{eq:dist-small}
    {\rm dist}(\vzero, \mF(\vy) + \gN_\gX(\vy)) \le {\rm res}(\vx).
\end{equation}
\end{prop}

\section{Main Result}

The main theorem below shows a new complexity upper bound of $T = \tilde \gO(\epsilon^{-1/p})$ for finding an $\eps$-solution such that ${\rm res}(\vx) \le \eps$, which is equivalent to the convergence rate ${\rm res}(\vx) =  \tilde \gO( T^{-p} )$ as claimed. It strictly improves the best-known convergence rate of $\mathcal{O}(T^{-(p+1)/2})$ given by the $p$th-order generalization of the NPE method \citep{monteiro2012iteration,bullins2022higher,lin2022perseus,huang2022approximation}.
 

\begin{thm}\label{thm:main}
Under Assumptions \ref{asm:X}-\ref{asm:pth-smooth}, for any integer $p \ge 2$, there exists an algorithm (Algorithm \ref{alg:inexact-halpern}) that can return an $\epsilon$-solution $\vx \in \gX$ satisfying ${\rm res}(\vx) \le \eps$ in the $p$th-order oracle complexity of 
\[
T = {\widetilde{\mathcal{O}}\Bigg(D \left(  \frac{L_{p}}{\epsilon}\right)^{1/p}\Bigg),}
\]
where $D = \left\|\vx_{0}-\vx^*\right\|$ is the distance of the initial point $\vx_0 \in \gX$ to the optimal solution $\vx^*$.
\end{thm}


\section{Halpern-NPE Achieves the $\tilde \gO(T^{-2})$ Rate for $p=2$}

In this section, we first prove Theorem \ref{thm:main} for $p=2$ by introducing a simple second-order method that achieves the fast convergence rate of $\tilde \gO(T^{-2})$. We apply a large-step inexact Halpern iteration on the resolvent operator, which is solved by a restarted version \citep{huang2022approximation,lin2022perseus} of the NPE method \citep{monteiro2012iteration}. We formally introduce our double-loop Halpern-NPE method algorithm in the following.

\subsection{Outer Loop: Inexact Halpern Iteration} \label{subsec:inext-halpern}


Halpern iteration \citep{lieder2020convergence,diakonikolas2020halpern} with coefficients $\beta_t = 1/(t+2)$ on the approximate resolvent $\vy_t \approx \mP_\eta(\vx_t)$ is:
\begin{equation} \label{eq:inexact-halpern}
\vx_{t+1} = \frac{1}{t+2} \vx_0 + \frac{t+1}{t+2} \vy_t, \qquad {\rm where} \quad \vy_t \approx \mP_\eta(\vx_t).
\end{equation}
This inexact Halpern iteration has been analyzed in \citep{diakonikolas2020halpern,alacaoglu2024revisiting,cai2024variance}. This iteration converges at the rate of ${\rm res}(\vx_T) = \gO( (\eta T)^{-1})$ if the approximation error of the resolvent is sufficiently small at every step.

\begin{lem}[{\citet[Theorem 2.1]{alacaoglu2024revisiting}}] \label{lem:inexact-halpern}
Under Assumptions \ref{asm:X}-\ref{asm:F}, if the resolvent approximation error satisifes 
\begin{equation} \label{eq:cond-delta-t}
   \|\vy_t - \mP_\eta(\vx_t) \| \le \delta_t:= \frac{\| \vx_t - \mP_\eta(\vx_t)\|}{98 \sqrt{t+2} \log (t+2) }, \qquad  t =0,\cdots,T-1,
\end{equation}
then the inexact Halpern iteration in \eqref{eq:inexact-halpern} guarantees a convergence rate of
\begin{equation} \label{eq:halpern-rate}
    {\rm res}(\vx_T) \le  \frac{4 \|\vx_0 - \vx^* \|}{ \eta (T+1)}.
\end{equation}
\end{lem}
The above result has been extensively used in first-order methods, including parameter-free methods \citep{diakonikolas2020halpern}, finite-sum problems \citep{cai2024variance}, and structured non-monotone problems \citep{alacaoglu2024revisiting}. In the high-order settings considered in this paper, we show that it allows a larger stepsize of $\eta = \gO(T / L_2)$ than the choice of $\eta = \gO(\sqrt{T}/ L_2)$ in the proximal point method \citep{monteiro2012iteration,bullins2022higher,lin2022perseus,huang2022approximation}. 

\subsection{Inner Loop: Approximating the Resolvent} \label{subsec:approx-resolvent-p2}
Recall that the resolvent in Definition \ref{dfn:resolvent} is \(\mP_\eta(\vx) = ({\rm Id} + \eta (\mF + \gN_\gX ))^{-1} (\vx)\). Equivalently, the proximal point $\vy=\mP_\eta(\vx)$ is the unique solution of
\begin{equation} \label{eq:optimality-resolvent}
    \vzero \in \mF(\vy)+ \frac{1}{\eta} (\vy-\vx)+\gN_{\gX}(\vy).
\end{equation}
For a fixed $\vx$, the above problem is equivalent to solving the MVI problem induced by the regularized operator $\mG(\vy) := \mF(\vy) + (\vy  -\vx) / \eta$, which is $\mu $-strongly monotone for $\mu = 1/\eta$:
\begin{equation} \label{eq:regularized-strongly-monotone}
    \langle \mG(\vy) -  \mG(\vy'), \vy - \vy' \rangle \ge \mu \|\vy - \vy' \|^2, \qquad \forall \vy,\vy' \in \gX.
\end{equation}
The complexity of obtaining an approximate solution follows from existing results for smooth strongly MVIs \citep{monteiro2012iteration,jiang2022generalized,huang2022approximation,ostroukhov2020tensor}.  For instance, we can use the restarted NPE/ARE method \citep{monteiro2012iteration,huang2022approximation}, which is essentially a second-order extragradient method \citep{korpelevich1976extragradient} with large adaptive stepsize $\tau_k = \gO(1/ (L_2 \| \vx_{k+1/2} - \vx_k  \|))$ \citep{monteiro2012iteration,huang2022approximation}. See Algorithm~\ref{alg:ARE-restart} for the complete procedure, whose guarantee is stated as follows.

\begin{lem}[{\citet[Theorem 3.2]{huang2022approximation}}] \label{lem:ARE-restart}
Let the operator $\mG: \gX \to \sR^d$ be $L_2$-second-order smooth satisfying \eqref{eq:F-pth-smooth} and $\eta^{-1}$-strongly monotone. Apply Algorithm \ref{alg:ARE-restart} with the parameters
\[
M = \Theta(L_2), \quad
K = \gO \left( (\eta L_2  R)^{2/3} \right), \quad S = \gO \left( \log \log (R/\delta) \right)
\]
can return a point $\vy_S \in \gX$ such that $\| \vy_S - \vy^* \| \le \delta$ in the total $p$th-order oracle complexity of $K\times S$, where $\vy^* = (\mG + \gN_\gX)^{-1}(\vzero)$ is the unique solution to the MVI induced by $\mG$ and $R = \|\vy_0 - \vy^* \|$.
\end{lem}

\begin{algorithm}[htbp]  
\caption{\textsf{Restarted-NPE}$(\mG,\vy_0,M, K,S)$}  \label{alg:ARE-restart}
\begin{algorithmic}[1] 
\State \textbf{for} $s = 0,\cdots, S-1$ 
\State \quad $\vy_{s,0} = \vy_s$ 
\State \quad \textbf{for} $k = 0,\cdots,K-1$
\State \quad \quad Perform a cubic regularized Newton step $\vy_{s,k+1/2} = \gT_\mG^2(\vy_{s,k};M)$ \vspace{1mm}
\State \quad \quad Set the adaptive stepsize $\tau_{s,k} = 1 / (M \| \vy_{s,k+1/2} - \vy_{s,k} \|)$) \vspace{1mm}
\State \quad \quad Perform an extragradient step \( \vy_{s,k+1} =\gT_\mG^1(\vy_{s,k};\tau_{s,k}^{-1}) \) \vspace{1mm}
\State \quad \textbf{end for}
\State \quad Restart at $\vy_{s+1}= \sum_{s=0}^{S-1} \tau_{s,k}~ \vy_{s,k+1/2}~ \big /  \sum_{s=0}^{S-1} \tau_{s,k}$
\State \textbf{end for} 
\State \textbf{return} $\vy_S$
\end{algorithmic}
\end{algorithm}

We remark that the above complexity of $K \times S$ can be reduced to $K + S$ with a refined analysis, as in \citet[Section 7.2]{jiang2022generalized} or \citet[Theorem 4.1]{huang2022approximation}. However, the slightly looser bound of $K \times S$ is enough for our subsequent analyses.


\subsection{Final Algorithm and Total Complexity}

\begin{algorithm}[htbp]  
\caption{\textsf{Halpern-NPE}$(\mF,\vx_0, \eta, T, M, \{K_t\}_{t=0}^{T-1},  \{S_t\}_{t=0}^{T-1})$}  \label{alg:inexact-halpern}
\begin{algorithmic}[1] 
\State \textbf{for} $t = 0,\cdots, T-1$ 
\State \quad Let $\mG_t(\vy) = \mF(\vy) + (\vy - \vx_t)/\eta$ and approximately solved $ \vzero \in \mG_t(\vy) + \gN_\gX(\vy)$ by
\[
\vy_t = \textsf{Restarted-NPE}(\mG_{t},\vx_t, M, K_t, S_t)
\]
\State \quad Perform the Halpern iteration
\[
\vx_{t+1} = \frac{1}{t+2} \vx_0 + \frac{t+1}{t+2} \vy_t
\]
\State \textbf{end for} 
\State \textbf{return} $\vx_T$
\end{algorithmic}
\end{algorithm}

By applying Algorithm \ref{alg:ARE-restart} to solve the resolvent operator in \eqref{eq:inexact-halpern}, we obtain our final high-order method to solve MVI problems in Algorithm \ref{alg:inexact-halpern}. Then, by appropriately selecting the stepsize $\eta = \Theta(T)$ as well as sub-solver parameter sequences $K_t$ and $S_t$, we can obtain the total complexity of Algorithm \ref{alg:inexact-halpern} as follows.

\begin{thm}[Theorem \ref{thm:main} for $p=2$] \label{thm:main-p-2}
Under Assumptions \ref{asm:X}-\ref{asm:pth-smooth} for $p = 2$, apply 
Algorithm \ref{alg:inexact-halpern} with parameters
\begin{align*}
T =& \gO \left( D \left(  \frac{L_{2}}{\epsilon}\right)^{1/2}\right), \quad
\eta = \Theta\left(\frac{T}{L_2 D} \right), \quad  M =\Theta(L_2), \\
K_t=&  \gO \left( \left( \frac{T}{t+1} \right)^{2/3} \right), \qquad S_t = \Theta( \log \log (t+2)),
\end{align*}
then the algorithm can ensure ${\rm res}(\vx_T) \le \eps$ in the total second-order oracle complexity of
\begin{equation} \label{eq:p2-complexity-bound}
    \gO(T \log \log T) = \tilde \gO( D \left( L_2/ \eps\right)^{1/2}),
\end{equation}
 where $D = \left\|\vx_{0}-\vx^*\right\|$ is the distance of the initial point $\vx_0 \in \gX$ to the optimal solution $\vx^*$.
\end{thm}

\begin{proof} 

Let us prove by induction that our parameter setting ensures the convergence rate of 
\begin{equation} \label{eq:induction-convergence-rate}
    {\rm res}(\vx_T): = \frac{\| \vx_T - \mP_\eta(\vx_T) \|}{\eta} \le  \frac{4 D}{ \eta (T+1)}.
\end{equation}
For $t=0$, we know from the non-expansiveness of $\mP_\eta$ and the fixed-point property $\vx^* = \mP_\eta(\vx^*)$ that
\begin{equation} \label{eq:induct-base-R0}
\| \vx_0 - \mP_\eta(\vx_0) \| \le \| \vx_0  - \mP_\eta(\vx^*) \| + \| \mP_\eta(\vx^*)  - \mP_\eta(\vx_0) \| \le 2 \| \vx_0 - \vx^* \| = 2D,    
\end{equation}
which established the induction base of \eqref{eq:induction-convergence-rate} by dividing $\eta$ on both sides. Now, we assume \eqref{eq:induction-convergence-rate} holds for all $t \le T-1$. Then, $R_t : = \| \vx_t - \mP_\eta(\vx_t) \|$, the initial distance of Algorithm \ref{alg:ARE-restart} satisfies that 
\begin{equation} \label{eq:Rt-rate}
    R_t \le \frac{4 D}{t+1}, \quad \forall t = 0,\cdots,T-1.
\end{equation}
Therefore, by Lemma \ref{lem:ARE-restart}, at each step the resolvent can be solved up to the approximation error $\delta_t$ in \eqref{eq:cond-delta-t} under the parameter setting \(M = \Theta(L_2)\), 
\[
K_t = \Theta \left( (\eta L_2 R_t)^{2/3} \right) = \Theta \left( \left( \frac{T R_t}{D} \right)^{2/3}  \right) = \gO \left( \left( \frac{T}{t+1} \right)^{2/3}  \right),
\]
and, by the definition of  $\delta_t = \Theta(R_t / (\sqrt{t+2} \log (t+2))$ in \eqref{eq:cond-delta-t}, 
\[
S_t = \Theta \left( \log \log (R_t/ \delta_t) \right) = \Theta \left( \log \log (t+2) \right).
\]
Therefore, the inexact condition in Lemma \ref{lem:inexact-halpern} is satisfied for $t \le T-1$, which implies the convergence rate of \eqref{eq:induction-convergence-rate} for $t =T$ by the lemma. This completes the induction.

Now, substituting our choices of $\eta$ in \eqref{eq:induction-convergence-rate} yields 
\[
{\rm res}(\vx_T) =\gO \left(  L_2 D^2/T^2 \right),
\]
which is smaller than $\epsilon$ by our setting of $T$. Finally, the total number of oracle calls can be bounded by
\begin{equation*} 
    \sum_{t=0}^{T-1} K_t S_t = \gO \left( \sum_{t=0}^{T-1} \left( \frac{T}{t+1} \right)^{2/3} \log \log T \right) = \gO(T \log \log T),
\end{equation*}
which is equivalent to the bound of \(\tilde \gO( D \left( L_2/ \eps\right)^{1/2})\) in \eqref{eq:p2-complexity-bound} by substituting the choice of $T$.
\end{proof}

This theorem shows that the proposed Halpern-NPE converges at a fast rate of $\tilde \gO(T^{-2})$, which improves both the classical rate of $\gO(T^{-1.5})$ by NPE \citep{monteiro2012iteration} and the rate of $\tilde \gO(T^{-1.75})$ via primal-dual A-NPE \citep{chen2025solving} for minimax problems. In the next section, we generalize this theorem to all $p \ge 2$ and achieve the upper bound of $\tilde \gO(T^{-p})$ claimed by the main Theorem \ref{thm:main}.

\section{Generalizing the Result to All $p \ge 2$}

To achieve the rate of $\tilde \gO(T^{-p})$ for all $p \ge 2$, a natural idea is to also use the Halpern iteration with a large stepsize $\eta = \Theta(T^{p-1})$ and reuse the analysis in the proof of Theorem \ref{thm:main-p-2}. To ensure that each subproblem is solvable in $\tilde \gO(1)$ costs like in \eqref{eq:tloglogt}, the sub-solver for the resolvent operator should converges at least $ \tilde \gO( \exp( -  \mu^{1/(p-1)}~T))$ for $\mu$-strongly monotone operators, or equivalently (under black-box reductions \citep{allen2016optimal,ostroukhov2020tensor,huang2022approximation,lin2022perseus}),  the $\tilde \gO(T^{-(p-1)})$ for monotone operators.

However, the classical convergence rate of $\gO(T^{-(p+1)/2})$ achieved by high-order NPE \citep{bullins2022higher,huang2022approximation} does not satisfy the above requirement when $p \ge 4$. In the following, we first introduce an Anchored Tensor Method (ATM) that achieves the required convergence rate of $\gO(T^{-p})$ in Section \ref{subsec:ATM}, then we further accelerate it with the Halpern iteration as in $p=2$ to achieve the $\tilde \gO(T^{-p})$ convergence rate in Section \ref{subsec:Halpenr-ATM}.

\subsection{ATM Achieves the $\gO(T^{-(p-1)})$ Rate} \label{subsec:ATM}

Recall Section \ref{subsec:approx-resolvent-p2} that approximating the resolvent is 
\(\mP_\eta(\vx) = ({\rm Id} + \eta (\mF + \gN_\gX ))^{-1} (\vx)\) is 
equivalent to solving the MVI problem induced by the $\eta^{-1}$-strongly monotone operator $\mG(\vy) := \mF(\vy) + (\vy  -\vx) / \eta$. We introduce a novel Anchored Tensor Method (ATM) that can solve the problem in the $p$th-order oracle complexity of \( \gO( \eta^{1/(p-1)})\). It also implies a complexity of $\gO(\eps^{-1/(p-1)})$ for finding an $\eps$-solution in the monotone case by applying the same algorithm on an $\eps$-regularized operator \citep[Lemma 4.1]{chen2026solving}.

\begin{algorithm}[htbp]  
\caption{\textsf{Anchored-Tensor-Method}$(\mG,\vy_0, \mu, M, R, K, S)$}  \label{alg:anchoring}
\begin{algorithmic}[1] 
\State \textbf{for} $k = 0,\cdots, K-1$ 
\State \quad Select regularization $\mu_k$ according to \eqref{eq:rj-muj}.
\State \quad Define the regularized operator $\mG_k(\vy) = \mG(\vy) + (\mu_k - \mu) (\vy - \vy_0)$
\State \quad Perform a tensor step $\vy_{k+1} = \gT_{\mG_k}^p (\vy_k;M)$
\State \textbf{end for}
\State \textbf{for} $s = K,\cdots, K+S-1$ 
\State \quad Perform a tensor step $\vy_{k+1} = \gT_{\mG}^p (\vy_k;M)$
\State \textbf{end for}
\State \textbf{return} $\vy_{K+S}$
\end{algorithmic}
\end{algorithm}

The starting point of the anchored tensor method is the following local convergence guarantee for the tensor step \citep{huang2022approximation,lin2022perseus}.

\begin{lem}[{\citet[Theorem 3.5]{lin2022perseus}}] \label{lem:local-perseus}
Let the operator $\mG: \gX \to \sR^d$ be $L_p$-$pth$-order smooth satisfying \eqref{eq:F-pth-smooth} and $\mu$-strongly monotone satisfying \eqref{eq:regularized-strongly-monotone}. There exists a numerical constant $C$ such that The tensor step $\vy^+ = \gT_\mG^p(\vy; C L_p)$ satisfies
\[
\| \vy^+ - \vy^* \| \le \sqrt{C_p \eta L_p} \| \vy - \vy^* \|^{(p+1)/2}, \quad C_p: = \frac{2^p (5p-2)}{p!},
\]
where $\vy^* = (\mG + \gN_\gX)^{-1}(\vzero)$ is the unique solution to the MVI induced by $\mG$.
\end{lem}

Define the local region radius $\rho(\mu)$ and contraction factor $\theta_p$ indicated by this lemma as
\[
\rho(\mu):= \frac{1}{2} \left( \frac{\mu}{C_p L_p} \right)^{1/(p-1)}, \qquad \theta_p:= 2^{-(p-1)/2} <1.
\]
Consequently, Lemma \ref{lem:local-perseus} implies hat
\begin{equation} \label{eq:y-plus-local}
    \| \vy - \vy^* \| \le \rho(\mu) \quad \Longrightarrow \quad \|\vy^+ - \vy^* \|\le  \theta_p \rho(\mu).
\end{equation}
It means that whenever $\vy$ enters the local region, the tensor update converges superlinearly to the minimizer. Based on this observation, we formally propose a two-phase Anchored Tensor Method in Algorithm \ref{alg:anchoring}:
\begin{enumerate}
    \item In the first phase with $K$ iterations, the algorithm iteratively performs the tensor step on the anchored/regularized operator \(\mG_k(\vy) := \mG(\vy) + (\mu_k-\mu)(\vy - \vy_0) \). The anchoring coefficient $\mu_k$ is chosen such that every tensor step lies in the local superlinear convergence region and $\mu_k$ decreases over iterations as the algorithm approaches the minimizer.
    \item At the beginning of the second phase, the anchoring coefficient $\mu_K$ has decreased to $\mu_K = \mu$ and the algorithm has entered the local region for the unregularized operator $\mG$. Therefore, the second phase reaches an $\delta$-solution in additional $S = \gO(\log \log \delta)$ iterations.
\end{enumerate}
To analyze the ATM method, we first give two lemmas for the anchored operator.

\begin{lem}[{\citet[Lemma A.1]{chen2024near}}]\label{lem:localization}
Let the operator $\mG$ with MVI solution $\vy^*$ satisfies Assumptions~\ref{asm:X}-\ref{asm:F}. For $\nu>0$, we define $\mG_\nu(\vy) = \mG(\vy) + \nu (\vy- \vy_0)$ be the regularized operator, and let $\vy_\nu^*$ be the unique solution of \(0\in \mF_{\nu}(\vy)+\gN_{\gX}(\vy)\). Then we have
\[
 {\left\|\vy_{\nu}^*-\vy_0 \right\|^{2}+\left\|\vy_{\nu}^*-\vy^*\right\|^{2}\leq\left\|\vy_0-\vy^*\right\|^{2}}
\]
\end{lem}


\begin{lem}[Stability of anchoring] \label{lem:sensitivity}
Under Assumptions \ref{asm:X}-\ref{asm:F}, for every $\nu_1 > \nu_2 > 0$,
\[
 \|\vy_{\nu_1}^* - \vy_{\nu_2}^* \|\leq \left( 1 - \frac{\nu_2}{\nu_1} \right)\|\vy_0-\vy^*\|.
\]
\end{lem}

Lemma \ref{lem:sensitivity} follows from the first-order optimality condition of $\vy_\nu^*$ and the monotonicity of $\mG$. The formal proof is contained in Appendix \ref{sec:proof-sensitivity}. With the help of these two lemmas, we carefully design the schedule of $\mu_k$ such that the local region $r_k$ contracts fast. Formally, let $C_p $ and $\theta_p$ be the constants implied by Lemma~\ref{lem:local-perseus}, we define
\begin{equation} \label{eq:rj-muj}
    r_k := \max \left\{ \rho(\mu), \left( \frac{1}{R} + \frac{k (1 - \theta_p)}{2R (p-1)} \right)^{-1}  \right\}, \qquad
 \mu_k := \rho^{-1}(r_k) =  C_p L_p (2r_k)^{p-1},
\end{equation}
where $R>0$ satisfies $\| \vy_0 - \vy^*  \| \le R$. By this definition, $r_k$ is decreasing from $r_0$ that satisfies $r_0 \le R$  to $\rho(\mu)$. Consequently, $\mu_j$ is decreasing from $\mu_0$ such that $\mu_0 \ge \rho^{-1}(R)$ to $\mu$. We summarize the properties of these sequences in the following.

\begin{lem} \label{lem:sequence-property}
The sequences $\{r_k \}_{k=0}^{K-1}$ and $\{ \mu_k\}_{k=0}^{K-1}$ defined in \eqref{eq:rj-muj} satisfy the following:
\begin{enumerate}
\item $r_{k+1} / r_k \ge (1 + \theta_p)/2$.
    \item $ 1- \mu_{k+1} / \mu_k \le (1- \theta_p) r_k / (2 R)$.
\end{enumerate}
\end{lem}

Lemma \ref{lem:sequence-property} follows from basic algebraic calculations, and the full proof is contained in Appendix \ref{sec:proof-rj-muj}. Now, we are ready to give the complexity bound of the ATM (Algorithm \ref{alg:anchoring}).

\begin{thm} \label{thm:ATM}
Let the operator $\mG: \gX \to \sR^d$ be $L_p$-$pth$-order smooth satisfying \eqref{eq:F-pth-smooth} and $\mu$-strongly monotone satisfying \eqref{eq:regularized-strongly-monotone}. Apply Algorithm \ref{alg:anchoring} with the parameters
\[
M = \Theta(L_p), \quad K = \gO\left( R (L_p/\mu)^{1/(p-1)} \right), \quad S = \gO( \log \log (R/ \delta) )
\]
can return a point $\vy_{K+S} \in \gX$ such that $\| \vy_{K+S} - \vy^* \| \le \delta$ in the total $p$th-order oracle complexity of $K+S$, where $\vy^* = (\mG + \gN_\gX)^{-1}(\vzero)$ is the unique solution to the MVI induced by $\mG$ and $R = \|\vy_0 - \vy^* \|$.
\end{thm}

\begin{proof}
Let us prove by induction that our parameter setting ensures that
\begin{equation} \label{eq:induction-local-region}
    \|\vy_k - \vy_k^*  \| \le r_k = \rho(\mu_k),
\end{equation}
where $\vy_k^*$ is the unique solution of the MVI induced by the anchored operator $\mG_k$ that satisfies \( 0 \in \mG_k(\vy) + \gN_\gX(\vy) \). For $k=0$, we know from Lemma \ref{lem:localization} that \( \|\vy_0  -\vy_0^* \| \le R \), which established the induction base since $r_0 = R$ by definition. Now, we assume \eqref{eq:induction-local-region} holds for the $k$th iteration. Then, by \eqref{eq:y-plus-local}, 
\[
\| \vy_{k+1} - \vy_k^* \| \le \theta_p r_k.
\]
Moreover, by Lemma \ref{lem:sensitivity} and Lemma \ref{lem:sequence-property} (part 2), we have
\[
\| \vy_{k+1}^* - \vy_{k}^* \| \le \left( 1 - \frac{\mu_{k+1}}{\mu_k} \right) R \le \frac{(1 - \theta_p) r_k}{2}.
\]
Combining the above two inequalities and using the triangle inequality, we obtain 
\[
\| \vy_{k+1} - \vy_{k+1}^* \| \le \| \vy_{k+1} - \vy_k^* \| + \| \vy_{k+1}^* - \vy_{k}^* \| \le \frac{(1 + \theta_p) r_k}{2} \le r_{k+1},
\]
where the last step uses Lemma \ref{lem:sequence-property} (part 1). This completes the induction. Consequently, all the tensor step in phase 1 lies in the local superlinear convergence region in Lemma \ref{lem:local-perseus}. We let the total number of iterations of phase 1 be the smallest number such that $\mu_k$ decays to $\mu$. By \eqref{eq:rj-muj}, such a $K$ satisfies
\[
K = \gO\left( R (L_p/\mu)^{1/(p-1)} \right).
\]
After phase 1 ends, we have $\mu_K = \mu$ and thus $\mG_K = \mG$. Since the initialization of phase 2 already lies in the local region for $\mG$, by \eqref{eq:y-plus-local}, phase 2 can find a $\delta$-solution in $S = \gO(\log \log (R/\delta)$ iterations.
\end{proof}

Treating the $\gO(\log \log (R/ \delta))$ factor as a constant, the above theorem shows that ATM achieves a new complexity upper bound of $\gO( (L_p/\mu)^{1/(p-1)})$ for $L_p$-$p$th-order smooth and $\mu$-strongly monotone operators. This improves the classical $\gO( (L_p/\mu)^{2/(p+1)})$ complexity achieved by high-order NPE \citep{monteiro2012iteration,bullins2022higher,lin2022perseus,huang2022approximation,adil2022optimal} for any $p \ge 4$. As we have discussed, it also implies a complexity of $\gO(\eps^{-1/(p-1)})$ for finding an $\eps$-solution in the monotone case by applying the same algorithm on an $\eps$-regularized operator \citep[Lemma 4.1]{chen2026solving}.

\subsection{Halpern-ATM Achieves the $\tilde \gO(T^{-p})$ Rate} \label{subsec:Halpenr-ATM}

\begin{algorithm}[htbp]  
\caption{\textsf{Halpern-ATM}$(\mF,\vx_0, \eta, T,  M, D, \{K_t\}_{t=0}^{T-1},  \{S_t\}_{t=0}^{T-1})$}  \label{alg:inexact-halpern-anchored}
\begin{algorithmic}[1] 
\State \textbf{for} $t = 0,\cdots, T-1$ 
\State \quad Let $\mG_t(\vy) = \mF(\vy) + (\vy - \vx_t)/\eta$ and approximately solved $ \vzero \in \mG_t(\vy) + \gN_\gX(\vy)$ by
\[
\vy_t = \textsf{Anchored-Tensor-Method}(\mG_t, \vx_{t}, 1/
\eta, M, R_t,K_t,S_t),
\]
where $R_t = 4 D/ (t+1)$ according to \eqref{eq:Rt-rate}.
\State \quad Perform the Halpern iteration
\[
\vx_{t+1} = \frac{1}{t+2} \vx_0 + \frac{t+1}{t+2} \vy_t
\]
\State \textbf{end for} 
\State \textbf{return} $\vx_T$
\end{algorithmic}
\end{algorithm}

The ATM introduced in the previous section achieves the complexity upper bound of $\tilde \gO( \eps^{-1/(p-1)} )$ for monotone problems. In this section, we further show an accelerated complexity of $\tilde \gO( \eps^{-1/p})$ by applying ATM to solve the resolvent operator in the Halpern iteration in Section \ref{subsec:inext-halpern}. We present the resulting method in Algorithm \ref{alg:inexact-halpern-anchored}. In light of Theorem \ref{thm:main-p-2}, we choose a large stepsize $\eta = \Theta(T^{p-1})$, which indicates that the outer Halpern iteration converges at the rate of $\gO( (\eta T)^{-1}) = \gO(T^{-p})$. We can then complete the proof by providing a global upper bound of the total oracle complexity via a similar analysis to the proof of Theorem~\ref{thm:main-p-2} in $p=2$. This leads to the following main theorem.

\begin{thm}[Full version of Theorem \ref{thm:main} ] \label{thm:main-full-version}
Under Assumptions \ref{asm:X}-\ref{asm:pth-smooth}, for any integer $p \ge 2$, apply 
Algorithm \ref{alg:inexact-halpern-anchored} with parameters
\begin{align*}
T =& \gO \left( D \left(  \frac{L_{p}}{\epsilon}\right)^{1/p}\right), \quad
\eta = \Theta\left(\frac{T^{p-1}}{L_p D^{p-1}} \right), \quad  M =\Theta(L_p), \\
K_t=&  \gO \left(  \frac{T}{t+1} \right), \qquad S_t = \Theta( \log \log (t+2)),
\end{align*}
then the algorithm can ensure ${\rm res}(\vx_T) \le \eps$ in the total $p$th-order oracle complexity of
\begin{equation} \label{eq:p-complexity-bound}
    \gO(T \log T ) = \tilde \gO( D \left( L_p/ \eps\right)^{1/p}),
\end{equation}
 where $D = \left\|\vx_{0}-\vx^*\right\|$ is the distance of the initial point $\vx_0 \in \gX$ to the optimal solution $\vx^*$.
\end{thm}

The proof can be done by replacing the sub-solver in $p=2$ with ATM for general $p \ge 2$ and following the analysis in the proof of Theorem~\ref{thm:main-p-2}. The formal proof is contained in Appendix \ref{sec:proof-main-thm}. 
This theorem shows that the proposed Halpern-ATM converges at a fast rate of $\tilde \gO(T^{-p})$, which improves both the classical rate of $\gO(T^{-2/(p+1)})$ by high-order NPE \citep{bullins2022higher,adil2022optimal,lin2022perseus,huang2022approximation} and the rate of $\tilde \gO(T^{-4/(3p+1)})$ via high-order primal-dual A-NPE \citep{chen2026solving} for minimax problems.

\section*{Acknowledgment} The abive rates are first proved with the assistance of AI. The initial resolvent sub-solver given by AI has triple loops, and the author simplified it with human-readable analyses to Algorithm \ref{alg:ARE-restart} for $p=2$ and the single-loop method Algorithm \ref{alg:anchoring} for $p \ge 2$. See Appendix \ref{sec:alg-AI} for the AI-provided algorithm and analysis.

\bibliographystyle{plainnat}
\bibliography{sample}

\newpage
\appendix
\section{Proof of Lemma \ref{lem:sensitivity}} \label{sec:proof-sensitivity}

\begin{proof}
The first-order optimality conditions of $\vy_{\nu_1}^*$ and $\vy_{\nu_2}^*$ indicate that $-\nu_1(\vy_{\nu_1}^* - \vy_0) \in \mG(\vy_{\nu_1}^*) + \gN_\gX(\vy_{\nu_1}^*)$ and $-\nu_2(\vy_{\nu_2}^* - \vy_0) \in \mG(\vy_{\nu_2}^*) + \gN_\gX(\vy_{\nu_2}^*)$. Then, the monotonicity of $\mG+\gN_{\gX}$ gives
\[
\left\langle-\nu_1(\vy_{\nu_1}-\vy_0)+\nu_2(\vy_{\nu_2}-\vx_0),\vy_{\nu_1} - \vy_{\nu_2} \right\rangle \ge 0.
\]
Using $\vy_{\nu_1}-\vy_0= (\vy_{\nu_1} - \vy_{\nu_2})+(\vy_{\nu_2}-\vy_0)$ and Lemma~\ref{lem:localization}, we obtain 
\[
(\nu_1-\nu_2) \|\vy_0-\vy^*\| \left\|\vy_{\nu_1} - \vy_{\nu_2}\right\| \ge \nu_1 \left\|\vy_{\nu_1} - \vy_{\nu_2} \right\|^{2}.
\]
\end{proof}

\section{Proof of Lemma \ref{lem:sequence-property}} \label{sec:proof-rj-muj}

\begin{proof}

Note that we only need to prove the inequalities hold before $r_k$ reaches $\rho(\mu)$, or, equivalently, before $\mu_k$ reaches $\mu$. Then, the truncated $r_k$ and $\mu_k$ after that trivially satisfy the same inequalities.
\paragraph{Proof of part 1.} From the definition of $r_k$ in \eqref{eq:rj-muj}, we have
\[
\frac{1}{r_{k+1}} - \frac{1}{r_k} \le \frac{1- \theta_p}{2 R(p-1)}.
\]
Rearranging, we have
\begin{equation} \label{eq:ratio-rj}
    \frac{r_{k+1}}{r_k} \ge \left( 1 + \frac{(1- \theta_p) r_k}{2R(p-1)} \right)^{-1} 
\end{equation}
Finally, using the inequality $(1+x)^{-1} \ge 1-x$ for $x \ge 0$ and the facts $r_k \le R$ and $p \ge 2$ gives 
\[
\frac{r_{k+1}}{r_k} \ge 1 - \frac{(1- \theta_p)r_k}{2R(p-1)}\ge 1 - \frac{1- \theta_p}{2} = \frac{1+ \theta_p}{2}.
\]
\paragraph{Proof of part 2.} From the definition of $\mu_k$ and $r_k$ in \eqref{eq:rj-muj}, we have
\[
1- \frac{\mu_{k+1}}{\mu_k} = 1 - \left(\frac{r_{k+1}}{r_k} \right)^{p-1}
\]
Using \eqref{eq:ratio-rj} and then applying the inequality $1 - (1+x)^{-(p-1)} \le (p-1)x$ for $x \ge 0$, we obtain that
\[
1- \frac{\mu_{k+1}}{\mu_j} = 1 - \left(\frac{r_{k+1}}{r_k} \right)^{p-1} \le 1 - \left( 1 + \frac{(1- \theta_p)r_k}{2R(p-1)} \right)^{p-1} \le \frac{(1- \theta_p)r_k}{2R}
\]

\end{proof}

\section{Proof of Theorem \ref{thm:main-full-version}} \label{sec:proof-main-thm}
\begin{proof} 
Following the proof of Theorem \ref{thm:main-p-2}, we show by induction that our parameter setting ensures the convergence rate of 
\begin{equation} \label{eq:induction-convergence-rate-again}
    {\rm res}(\vx_T): = \frac{\| \vx_T - \mP_\eta(\vx_T) \|}{\eta} \le  \frac{4 D}{ \eta (T+1)}.
\end{equation}
The induction base of $t=0$ follows from \eqref{eq:induct-base-R0} by dividing $\eta$ on both sides. Now, we assume \eqref{eq:induction-convergence-rate-again} holds for all $t \le T-1$. Then, $R_t : = \| \vx_t - \mP_\eta(\vx_t) \|$, the initial distance of Algorithm \ref{alg:anchoring} satisfies 
\[
  R_t \le \frac{4 D}{t+1}, \quad \forall t = 0,\cdots,T-1.
\]
Therefore, by Theorem \ref{thm:ATM}, at each step the resolvent can be solved up to the approximation error $\delta_t$ in \eqref{eq:cond-delta-t} under the parameter setting \(M = \Theta(L_p)\), 
\[
K_t = \Theta \left( R_t (\eta L_p)^{1/(p-1)} ) \right) = \Theta \left(  \frac{T R_t}{D}   \right) = \gO \left(\frac{T}{t+1}   \right),
\]
and, by the definition of  $\delta_t = \Theta(R_t / (\sqrt{t+2} \log (t+2))$ in \eqref{eq:cond-delta-t}, 
\[
S_t = \Theta \left( \log \log (R_t/ \delta_t) \right) = \Theta \left( \log \log (t+2) \right).
\]
Therefore, the inexact condition in Lemma \ref{lem:inexact-halpern} is satisfied for $t \le T-1$, which implies the convergence rate of \eqref{eq:induction-convergence-rate-again} for $t =T$ by the lemma. This completes the induction.

Now, substituting our choices of $\eta$ in \eqref{eq:induction-convergence-rate-again} yields 
\[
{\rm res}(\vx_T) =\gO \left(  L_p D^p/T^p \right),
\]
which is smaller than $\epsilon$ by our setting of $T = \gO( D \left( L_p/ \eps\right)^{1/p})$. Finally, the total number of oracle calls can be bounded by
\begin{equation} \label{eq:tloglogt}
    \sum_{t=0}^{T-1} K_t + S_t = \gO \left( \sum_{t=0}^{T-1} \left( \frac{T}{t+1} + \log \log (t+2)\right)\right) = \gO(T \log T),
\end{equation}
which is equivalent to the bound of \(\tilde \gO( D \left( L_p/ \eps\right)^{1/p})\) in \eqref{eq:p-complexity-bound} by substituting the choice of $T$.
\end{proof}

\section{Initial Algorithm and Analysis by AI} \label{sec:alg-AI}

The initial resolvent sub-solver given by AI has triple loops whose full procedure is presented in Algorithm~\ref{alg:anchoring-triple-loop}. The initial proof by AI is also accessible through the following link:
\begin{center}
    https://chatgpt.com/share/6a6c372c-6efc-83e8-a378-c6aa1ea3b527
\end{center}
The authors simplified the algorithm and its proof by using the \textsf{Restarted-NPE} algorithm \citep{huang2022approximation} for $p=2$ and the single-loop \textsf{Anchored-Tensor-Method} introduced in Algorithm \ref{alg:anchoring} for general $p \ge 2$.

\begin{algorithm}[htbp]  
\caption{\textsf{Triple-Looped-Anchored-Tensor-Method}$(\mG,\vy_0,R,S,K)$}  \label{alg:anchoring-triple-loop}
\begin{algorithmic}[1] 
\State $R_0 = R$ 
\State \textbf{for} $s = 0,\cdots,S-1$ 
\State \quad $\vy_{s,0} = \vy_s$
\State \quad \textbf{for} $k = 0,\cdots, K-1$ 
\State \quad \quad Define the anchoring coefficient
\[
\mu_k = \frac{L_p R_s^{p-1}}{(h_p+ k \Delta_p)^{p-1}}, \quad {\rm where} \quad h_p = \frac{1}{2} c_p^{-1/(p-1)}, ~~ \Delta_p = \frac{h_p}{8 (p-1)}, ~~ c_p = \frac{2^p (5p-2)}{p!}.  
\]
\State \quad \quad Define the regularized operator $\mG_k(\vy) = \mG(\vy) + \mu_k (\vy - \vy_{s})$
\State \quad \quad $\vy_{s,k,0} = \vy_{s,k}$
\State \quad \quad \textbf{for} $j = 0,\cdots,5$
\State \quad \quad \quad Perform one tensor step $ \vy_{s,k,j+1} = \gT^p_{\mG_k}(\vy_{s,k,j}; 5L_p)$
\State \quad \quad \textbf{end for}
\State \quad \quad $\vy_{s,k+1} = \vy_{s,k,6}$
\State \quad \textbf{end for}
\State \quad Change the anchor point to $\vy_{s+1} = \vy_{s,K}$ and shrink the region by $R_{s+1} = R_s/4$
\State \textbf{end for}
\State \textbf{return} $\vy_S$
\end{algorithmic}
\end{algorithm}


\end{document}